\documentclass[11pt,a4paper]{article}
\title{\bf Weighted Littlewood--Paley inequalities for heat flows in $\RCD$ spaces}
\author{Huaiqian Li\footnote{Email: {\color{blue}huaiqianlee@gmail.com} }\vspace{3mm}\\
{\footnotesize Center for Applied Mathematics, Tianjin University,
 Tianjin 300072, P. R. China}
}
\date{}
\usepackage{amssymb,amsmath,amsfonts,amsthm,color,mathrsfs}

\def\R{\mathbb{R}}
\def\E{\mathbb{E}}
\def\P{\mathbb{P}}
\def\D{\mathbb{D}}

\def\L{\mathcal{L}}

\def\d{\textup{d}}
\def\D{\textup{D}}

\def\CD{\textup{CD}}

\def\MCP{\textup{MCP}}
\def\Ric{\textup{Ric}}

\def\supp{\textup{supp}}
\def\RCD{\textup{RCD}}
\def\vol{\textup{vol}}
\def\<{\langle}
\def\>{\rangle}
\def\Proof.{\noindent{\bf Proof. }}

\def\loc{\textup{loc}}

\def\newdot{{\kern.8pt\cdot\kern.8pt}}

\newtheorem{theorem}{Theorem}[section]
\newtheorem{lemma}[theorem]{Lemma}
\newtheorem{corollary}[theorem]{Corollary}
\newtheorem{proposition}[theorem]{Proposition}

\newtheorem{definition}[theorem]{Definition}
\theoremstyle{definition}\newtheorem{remark}[theorem]{Remark}

\begin{document}
\allowdisplaybreaks
\maketitle
\makeatletter 
\renewcommand\theequation{\thesection.\arabic{equation}}
\@addtoreset{equation}{section}
\makeatother 

\begin{abstract}
We establish inequalities on vertical Littlewood--Paley square functions for heat flows in the weighted $L^2$ space over metric measure spaces satisfying the $\RCD^\ast(0,N)$ condition with $N\in [1,\infty)$ and the maximum volume growth assumption. In the noncompact setting, the later assumption can be removed by showing that the volume of the ball growths at least linearly. The estimates are sharp on the growth of the 2-heat weight and the 2-Muckenhoupt weight considered. The $p$-Muckenhoupt weight and the  $p$-heat weight are also compared for all $p\in(1,\infty)$.
\end{abstract}



\section{Introduction}\hskip\parindent
The Littlewood--Paley inequality in $\R^n$ is originated from the $L^p$ boundedness  for all $1<p<\infty$ of the Littlwood--Paley $g$-function (which was introduced first by Littlewood and Paley \cite{LP1931} in $\R$ to study the dyadic decomposition of Fourier series); see \cite{St1958} or \cite[Chapter IV, Theorem 1]{St1970}. There are numerous studies and extensions on this result, and we are more concerned with the vertical (i.e., derivative with respect to the spacial variable) Littlewood--Paley square functions  in curved spaces. Let $M$ be a
complete   Riemannian manifold with  volume measure $\vol$, the non-negative Laplace--Beltrami operator $\Delta$, and the gradient operator $\nabla$. Denote  $(e^{-t\Delta})_{t\ge0}$ and $(e^{-t\sqrt{\Delta}})_{t\ge0}$ the heat flow and Poisson flow, respectively.  For every $f\in C_c^\infty(M)$, the vertical
Littlewood--Paley $\mathcal{H}$-function and $\mathcal{G}$-function are
defined respectively by
\begin{equation*}\label{cla-H}
\mathcal{H}(f)(x)=\Big(\int_0^\infty |\nabla e^{-t\Delta} f|^2(x)\,\d t\Big)^{1/2},
\end{equation*}
 and \begin{equation*}\label{cla-G}
\mathcal{G}(f)(x)=\Big(\int_0^\infty t|\nabla e^{-t\sqrt{\Delta}} f|^2(x)\,\d t\Big)^{1/2},
\end{equation*}
for every $x\in M$, where $|\cdot|$ is the norm in the tangent space induced by the Riemannian distance. The operator $\mathcal{H}$ (resp. $\mathcal{G}$) is said to be bounded in $L^p(M,\vol)$
for any $p\in (1,\infty)$,
if there exists a positive constant $C_p$ such that, for any $f\in C_c^\infty(M)$,
\begin{equation}\label{lp-bound}
\|\mathcal{H}(f)\|_{L^p(M,\vol)}\le C_p\|f\|_{L^p(M,\vol)}\quad (\mbox{resp. }\|\mathcal{G}(f)\|_{L^p(M,\vol)}\le C_p\|f\|_{L^p(M,\vol)}).
\end{equation}
For $1<p\leq2$, no additional assumptions on the complete and noncompact Riemannian manifold $M$ are needed for the boundedness of $\mathcal{H}$ and $\mathcal{G}$ in $L^p(M,\vol)$; see e.g. \cite[Theorem 1.2]{CDD}. However, for $2<p<\infty$, much stronger assumptions are need; for instance, see  \cite[Proposition 3.1]{CD2003} for the condition on the control of the gradient of the semigroup by the semigroup applied to the gradient, i.e., $|\nabla e^{-t\Delta}f|^2\leq Ce^{-t\Delta}|\nabla f|^2$ for any $f\in C_c^\infty(M)$. See also \cite{Stein70,Mey,Mey1981,Lou1987,LiH2017+} for other related studies.

In the other aspect, it is well known that many classical operators from harmonic analysis are bounded in the weighted $L^p$ space for all $1<p<\infty$, where the ``weight'' is referred to a $p$-Muckenhoupt weight or commonly called an $A_p$ weight, i.e., a non-negative locally integrable function satisfying the Muckenhoupt condition (see \cite{Muckenhoupt72} or Definition \ref{Muckenhoupt-weight} below). One of the  important questions is to find  sharp dependence on the growth of the $2$-Muckenhoupt weight $w$;
more precisely, given an operator $S: L^2_w(\R^n)\rightarrow L^2_w(\R^n)$, prove
\begin{equation}\label{T-weight}
\|S(f)\|_{L^2_w(\R^n)}\leq C(n,S)\phi\left(\|w\|_{A_2(\R^n)}\right)\|f\|_{L^2_w(\R^n)},
\end{equation}
where $\phi:\R_+\rightarrow\R_+$ is some function describing the optimal growth of $\|w\|_{A_2(\R^n)}$,  $C(n,S)$ denotes a constant, and $\|\cdot\|_{A_2(\R^n)}$ is defined in Definition \ref{Muckenhoupt-weight} below. The problem \eqref{T-weight} was first studied  by Buckley \cite{Buckley1993} and solved for the Hardy--Littlewood maximal operator. Then Petermichl and her coauthors proved \eqref{T-weight}  for the Beurling--Ahlfors operator, the Hilbert transform, the Riesz transform and the Haar shift; see \cite{PV2002,Petermichl2007,Petermichl2008,LPR2010}. Refer to \cite{CMP2010,CMP2012} for simplified proofs for Haar shifts.  Later, Hyt\"{o}nen \cite{Hyt2012} proved \eqref{T-weight} for the general Calder\'{o}n--Zygmund operator; see also \cite{HPTV2014} and \cite{Lerner2013} for  simplified proofs. In a very recent work \cite{BO2016}, by establishing  sharp weighted $L^2$ martingale inequalities, Ba\~{n}uelos and Osekowski proved \eqref{T-weight} for the dyadic square function, as well as the weighted version of \eqref{lp-bound}.

Motivated by \cite{BO2016}, we are going to establish weighted $L^2$ versions of vertical Littlewood--Paley square functions corresponding to the heat flow in the $\RCD$ space, which is presented in Section 3 below. In Section 2, we recall the definition of $\RCD$ spaces and some known results. In Section 4, we compare the $p$-heat weight and the $p$-Muckenhoupt weight, which is motivated by \cite[Section 3]{PV2002}. In Section 5, we study the smooth Riemannian manifold setting as a typical example, and finally we remark that similar results should be established on a large class of sub-Riemannian manifolds.

\section{Preliminaries}\hskip\parindent
In this section, we present some notions and know results; refer to \cite{AmbrosioGigliSavare2011b,AmbrosioGigliSavare2012,eks2013,gi2012} for more details.

\subsection{$\RCD$ spaces}\hskip\parindent
Throughout this work, $(M,d)$ will always denote a complete and separable metric space. Let $C([0, 1],M)$ be the Banach space of continuous curves from $[0, 1]$ to $M$ equipped with the supremum norm. For every $t\in [0, 1]$, recall that the evaluation map $e_t :C([0, 1],M) \rightarrow M$ is defined by
$$e_t(\gamma)=\gamma_t,\quad\mbox{for any }\gamma\in C([0, 1],M).$$

Let $q\in [1,\infty]$. A curve $\gamma: [0,1] \rightarrow M$ is said to be absolutely continuous, denoted by $\gamma\in AC_q([0,1],M)$, if there exists a function $h\in L^q([0,1])$ such that,
\begin{eqnarray}\label{ac}
d(\gamma_s,\gamma_t)\leq \int_s^t h(r)\,\d r,\quad\mbox{for any }0\leq s<t\leq 1.
\end{eqnarray}
If $\gamma\in AC_q([0,1];M)$, then it can be proved that the metric slope
$$\lim_{\epsilon\rightarrow 0}\frac{d(\gamma_{r+\epsilon},\gamma_r)}{|\epsilon|},$$
denoted by $|\dot{\gamma}_r|$, exists for  a.e. $r\in [0,1]$, belongs to $L^q([0,1])$,
and it is the minimal function $h$ such that \eqref{ac} holds (see Theorem 1.1.2 in \cite{AmbrosioGigliSavare2005} for the proof).
For every $\gamma\in C([0,1],M)$, we use the notation $\int_0^1 |\dot{\gamma}_r|^q\,\d r$, which may be $+\infty$ if $\gamma$ is
not absolutely continuous.

Endow $(M,d)$ with a non-negative Radon  measure $\mu$ with full topology support. We call the triple $(M,d,\mu)$ a metric measure space.

We recall first the notions of test plan and Sobolev class; see \cite{AmbrosioGigliSavare2014,gi2012} for more details.
\begin{definition}
A probability measure $\pi$ on $C([0, 1],M)$ is called a test plan if, there exists a positive constant $C$ such that
$$(e_t)_\sharp{ \pi}\le C\mu,\quad\mbox{for any }t\in [0,1],$$
and
$$\int \int_0^1|\dot{\gamma}_t|^2\,\d t\,\d\pi(\gamma)<\infty,$$
where $(e_t)_\sharp{ \pi}(E):=\pi(e_t^{-1}(E))$ for every Borel subset $E$ of $M$.
  \end{definition}

\begin{definition} \label{sobolev}
The Sobolev class $S^2(M):=S^{2}(M,d,\mu)$  is the space of all Borel functions $h: M\rightarrow \R$, for which there exists a non-negative function $f\in L^2(M)$ such that, for each test plan $\pi$, it holds
\begin{equation}\label{curve-sobolev}
\int |h(\gamma_1)-h(\gamma_0)|\,\d\pi(\gamma)\leq \int \int_0^1 f(\gamma_t)|\dot{\gamma}_t|\,\d t\,\d\pi(\gamma).
\end{equation}
\end{definition}
It turns out that for each  $h\in S^2(M)$, there exists a unique minimal function $f$
in the $\mu$-a.e. sense such that \eqref{curve-sobolev} holds. The minimal function $f$ is represented by $|\nabla  h|_w$
and called the minimal weak upper gradient of $h$. See e.g. \cite{AmbrosioGigliSavare2014}.

The Sobolev space $W^{1,2}(M):=W^{1,2}(M,d,\mu)$ is defined as $S^2(M)\cap L^2(M)$, which is a Banach space with the norm
$$\|f\|_{W^{1,2}(M)}:=\Big(\|f\|^2_{L^2(M)}+\| |\nabla  f|_w\|_{L^2(M)}^2\Big)^{1/2},$$
but, in general, not a Hilbert space.

Now we recall the definition of the so-called reduced curvature-dimension condition $\CD^*(K,N)$, which is first introduced in \cite{BacherandSturm2010} and it is a modification of the curvature-dimension condition $\CD(K,N)$ introduced independently by Lott--Villani \cite{LV2009} and Sturm \cite{Sturm2006a,Sturm2006b}. In particular, $\CD(0,N)$ and $\CD^\ast(0,N)$ coincide with each other.

Let $K,N\in\R$ with $N\geq1$. For every $(t,\theta)\in [0,1]\times [0,\infty)$, define
\begin{equation*}\label{distortion2}
\sigma_{K,N}^{(t)}(\theta)=
\begin{cases}
 \frac{\sinh\big(t\theta\sqrt{-K/N}\big)}{\sinh\big(\theta\sqrt{-K/N}\big)} ,\quad &{\mbox{if }K\theta^2< 0\mbox{ and }N>1},\\
t,\quad &{\hbox{if }K\theta^2=0,\mbox{ or if }K\theta^2<0\mbox{ and }N=1},\\
 \frac{\sin\big(t\theta\sqrt{K/N}\big)}{\sin\big(\theta\sqrt{K/N}\big)} ,\quad &{\hbox{if }0<K\theta^2<N\pi^2},\\
+\infty,\quad &{\hbox{if }K\theta^2\geq N\pi^2}.
\end{cases}
\end{equation*}
Let $\mathcal{P}_b(M)$ denote the class of Borel probability measures on $(M,d)$ with bounded support. Given two metric measure spaces $(X_1,d_1,\nu_1)$ and $(X_2,d_2,\nu_2)$, we say that a measure $\gamma$ on the product space $X_1\times X_2$ is a coupling of $\nu_1$ and $\nu_2$ if $$\gamma(A\times X_2)=\nu_1(A),\quad\gamma(X_1\times B)=\nu_2(B),$$
for all Borel subsets $A$ of $X_1$ and $B$ of $X_2$.
\begin{definition}\label{CD-star}
Let $K\in \R$ and $N\in [1,\infty)$. We say that the metric measure space $(M,d,\mu)$ is a $\CD^\ast(K,N)$ space, if for every pair $\eta_0,\eta_1\in \mathcal{P}_b(M)$ with $\eta_i=\rho_i\mu$, $i=0,1$, there exists an optimal coupling $\pi$ of $\eta_0$ and $\eta_1$ such that
\begin{eqnarray*}\label{CD-star-1}
&&\int_X \rho_t^{1-\frac{1}{N'}}\,\d\mu\cr
 &\geq& \int\left[\sigma_{K,N'}^{(1-t)}(d(\gamma_0,\gamma_1))\rho_0^{-\frac{1}{N'}}(\gamma_0) + \sigma_{K,N'}^{(t)}(d(\gamma_0,\gamma_1))\rho_1^{-\frac{1}{N'}}(\gamma_1) \right]\,\d\pi(\gamma),
\end{eqnarray*}
for all $t\in [0,1]$ and all $N'\geq N$, where $\rho_t$ denotes the Radon--Nikodym derivative $\frac{\d (e_t)_{\#}\pi}{\d\mu}$ for every $t\in [0,1]$.
\end{definition}

In order to rule out Finsler structures, the Riemannian curvature-dimension condition ($\RCD$ for short) is introduced in \cite{AmbrosioGigliSavare2011b,agmr2015} (with $N=\infty$) and then in \cite{gi2012,eks2013} (including $N<\infty$), which is more restrictive than the reduced curvature-dimension condition  $\CD^*(K,N)$ by requiring additionally the Banach space $W^{1,2}(M)$ to be a Hilbert space.

\begin{definition}
Let $K\in \R$ and $N\in [1,\infty)$. We say that a metric measure space $(M,d,\mu)$ is an $\RCD^*(K,N)$ space if it is a $\CD^*(K,N)$ space and $W^{1,2}(M)$ is a Hilbert space.
\end{definition}
Typical examples of $\RCD^\ast(K,N)$ spaces are complete weighted Riemannian manifolds satisfying the Bakry--Emery curvature-dimension condition (see \cite{BakryEmery1985}), as well as their limit spaces in the measured Gromov--Hausdorff sense (see \cite{Sturm2006a,Sturm2006b,LV2009}), Alexandorv spaces (with curvature bounded from below) (see \cite{Petr2011,ZZ2011}), and so on.

We recall the volume comparison property (see  \cite[Theorem 6.2]{BacherandSturm2010}) which will be applied to the proof of main results below.
\begin{proposition}\label{vol-comparison}
Let $(M,d,\mu)$ be  a $\CD^*(0,N)$ space with $N\in [1,\infty)$. For any $x\in M$ and any $r,R\in(0,\infty)$ with $R\geq r$,
$$\frac{\mu\big(B(x,R)\big)}{R^N}\leq \frac{\mu\big(B(x,r)\big)}{r^N}.$$
\end{proposition}

\subsection{Martingale inequalities}\hskip\parindent
Let $(\Omega,\mathcal{F},\mathbb{P})$ be a complete probability space equipped with a filtration $(\mathcal{F}_t)_{t\geq0}$, a nondecreasing right continuous family of sub-$\sigma$-fields of $\mathcal{F}$ such that $\mathcal{F}_0$ contains all the events with probability 0. Fix $T\in(0,\infty]$. Let $X=(X_t)_{t\geq0}$ be an adapted and uniformly integrable martingale having continuous path, and $\langle X\rangle=(\langle X\rangle_t)_{ t\geq0}$ be the quadratic variation process. Let  $Y=(Y_t)_{t\geq0}$ be a non-negative, uniformly integrable martingale with continuous path such that $Y_0=\mathbb{E}(Y_T)$. For $1<p<\infty$, following Izumisawa and Kazamaki in \cite{IzumisawaKazamaki1977}, we say that $Y$ satisfies the Muckenhoupt condition $A_p^{mart}$ if
$$\|Y\|_{A_p^{mart}}:=\sup_{0\leq t\leq T}\Big\|\Big(\mathbb{E}\Big[\Big(\frac{Y_t}{Y_T}\Big)^{1/(p-1)}\Big|\mathcal{F}_t\Big]\Big)^{p-1}\Big\|_{L^\infty(\mathbb{P})}<\infty,$$
where $Y_t=\mathbb{E}(Y_T|\mathcal{F}_t)$. The process $Y$ gives rise to a probability measure $\mathbb{Q}$ defined by
$$\mathbb{Q}(A)=\int_A Y_T\,\d\mathbb{P},\quad A\in\mathcal{F}_T,$$
and hence, it can be regarded as a weight.

Now we adapt results from \cite[Section 4]{BO2016} (see \cite[Theorem 1]{DomPet2017} for a more general result on Hilbert space valued differentially subordinate martingales, as well as \cite{DomPet2016} for sharp weighted $L^p$ estimate on the maximal
function of adapted uniformly integrable c\`{a}dl\`{a}g Hilbert space valued martingales) in the next theorem, which is one of the key tools to establish our main results.
\begin{theorem}\label{L2inequality}
Fix $T\in (0,\infty]$. Let $X=(X_t)_{t\geq0}$ be an adapted, real valued and uniformly integrable martingale with continuous path, and $Y=(Y_t)_{ t\geq0}$ be a non-negative, uniformly integrable martingale with continuous path. Suppose that $X$ is bounded in $L^2(\mathbb{Q})$ and $Y$ satisfies the Muckenhoupt condition $A_2^{mart}$. Then
\begin{equation}\label{L2-1}
\|X_T\|_{L^2(\mathbb{Q})}\leq \big(80\|Y\|_{A_2^{mart}}\big)^{1/2}\|\langle X\rangle_T^{1/2}\|_{L^2(\mathbb{Q})},
\end{equation}
and
\begin{equation}\label{L2-2}
\|\langle X\rangle_T^{1/2}\|_{L^2(\mathbb{Q})}\leq \inf_{1<r<2}\Big(\frac{r}{2-r}\|Y\|_{A_r^{mart}}\Big)^{1/2}\|X_T\|_{L^2(\mathbb{Q})}.
\end{equation}
Moreover,
\begin{equation}\label{L2-3}
\|\langle X\rangle_T^{1/2}\|_{L^2(\mathbb{Q})}\leq 2^{7/4}\|Y\|_{A_2^{mart}}\|X_T\|_{L^2(\mathbb{Q})}.
\end{equation}
\end{theorem}

\section{Weighted $L^2$ inequalities in $\RCD$ spaces}\hskip\parindent
In this section, let $(M,d,\mu)$ be an $\RCD^\ast(K,N)$ space with $K\in\R$ and $N\in [1,\infty)$. Then $(M,d)$ is a locally compact length space, and indeed, a geodesic and proper space (i.e., every bounded and closed subset is compact); see e.g. \cite[Remark 6.2]{AMS2016}. For any function $f\in W^{1,2}(M)$, define
$$\D(f)=\int_M |\nabla f|_w^2\,\d\mu.$$
Then, for $f,g\in W^{1,2}(M)$, let
$$\D(f,g)=\frac{1}{4}[\D(f+g)-\D(f-g)].$$
By the  parallelogram law (due to the Hilbert structure of $W^{1,2}(X)$), we immediately derive that
\begin{eqnarray*}
\D(f,g)&=&\frac{1}{4}\Big[\int_M|\nabla(f+g)|_w^2\,\d\mu  - \int_M|\nabla(f-g)|_w^2\,\d\mu\Big]\\
&=&\int_M \Gamma(f,g)\,\d\mu,
\end{eqnarray*}
where
$$\Gamma(f,g)(x):= \lim_{\epsilon\downarrow0} \frac{|\nabla (g+\epsilon f)|_w^2(x)-|\nabla g|_w^2(x)}{2\epsilon},\quad\mbox{for }\mu\mbox{-a.e. }x\in M,$$
and the limit is taken in $L^1(M)$, which represents right the \emph{carr\'{e} du champ} (see e.g. \cite{BakryEmery1985}). It is known that $(\D, W^{1,2}(M))$ is a strongly local and regular Dirichlet form; see \cite[Section 4.3]{AmbrosioGigliSavare2011b} and \cite{AmbrosioGigliSavare2012}.  Recall that the Dirichlet form  $(\D,\mathscr{F})$ is
called regular if $W^{1,2}(M)\cap C_c(M)$ is dense both in $W^{1,2}(M)$ (with respect to the
$D_1^{1/2}$-norm defined by $\D_1(f,g)=\D(f,g)+\int_M fg\,\d\mu$) and in $C_c(M)$ (with respect to the supremum norm), and $(\D,\mathscr{F})$ is
called strongly local if for any $f,g\in W^{1,2}(M)$, $(f+a)g=0$ $\mu$-a.e. in  $M$ for some $a\in\R$, then $\D(f,g)=0$. See e.g. \cite{FOT2011}.

Denote by $(P_t)_{t\geq0}$ and $\Delta$ the heat flow and the infinitesimal generator, respectively, corresponding to $(\D,W^{1,2}(M))$. Let $(p_t)_{t\geq0}$ be the heat kernel corresponding to $(P_t)_{t\geq0}$. Then it is symmetric, i.e., for every $t>0$, $p_t(x,y)=p_t(y,x)$ for all $(x,y)\in M\times M$, and stochastically complete, i.e.,
\begin{eqnarray}\label{sto.com.}
\int_Mp_t(x,y)\, \d\mu(y)=1,\quad\mbox{for all }t>0\mbox{ and for all }x\in M.
\end{eqnarray}
Moreover, the author with Jiang and Zhang obtained the following heat kernel upper and lower bounds (see \cite[Theorem 1.1]{JLZ2014}).
\begin{proposition}\label{gassian-bound}
Let $(M,d,\mu)$ be an  $\RCD^*(0,N)$ space with $N\in [1,\infty)$.
Then, there exists a positive constant $C$ depending on $N$ such that
\begin{equation}\label{gassian-bound-1}
\frac{1}{C\mu(B(x,\sqrt t))}\exp\Big\{-\frac{d^2(x,y)}{3t}\Big\}\le p_t(x,y)\le \frac{C}{\mu(B(x,\sqrt t))}\exp\Big\{-\frac{d^2(x,y)}{5t}\Big\},
\end{equation}
for any $t>0$ and any $x,y\in M$.
\end{proposition}

For $f\in C_c(M)$ and $x\in M$, define the Littlewood--Paley $\mathcal{H}$-function and $\mathcal{H}_\ast$-function by
$$\mathcal{H}(f)(x)=\Big(\int_0^\infty|\nabla P_tf|_w^2(x)\,\d t\Big)^{1/2},$$
and
$$\mathcal{H}_\ast(f)(x)=\Big(\int_0^\infty\int_M|\nabla P_tf|_w^2(y)p_t(x,y)\,\d\mu(y)\d t\Big)^{1/2},$$
respectively. Then, it is easy to know that both $\mathcal{H}(f)$ and $\mathcal{H}_\ast(f)$ are bounded in $L^2(M,\mu)$. Indeed, on the one hand,
\begin{eqnarray*}
\int_M|\mathcal{H}(f)(x)|^2\,d\mu(x)&=&\int_M\int_0^\infty|\nabla P_tf|_w^2(x)\,\d t \d\mu(x)\\
&=&\int_0^\infty\Big(\int_M -\Delta(P_tf) P_tf\,\d\mu\Big)\,\d t\\
&=&-\int_0^\infty\int_M \Big(\frac{\d}{\d t}P_tf\Big)P_tf\,\d\mu \d t\\
&=&-\frac{1}{2}\int_0^\infty\int_M \frac{\d}{\d t}(P_tf)^2\,\d\mu \d t\\
&\leq&\frac{1}{2}\int_M f^2\,\d\mu,
\end{eqnarray*}
and on the other hand, by the symmetry $p_t(x,y)=p_t(y,x)$ and the stochastic completeness,
\begin{eqnarray*}
\int_M|\mathcal{H}_\ast(f)(x)|^2\,d\mu(x)&=&\int_M\int_0^\infty\int_M|\nabla P_tf|_w^2(y)p_t(x,y)\,\d\mu(y)\d t \d\mu(x)\\
&=&\int_0^\infty\int_M|\nabla P_tf|_w^2(y)\Big(\int_M p_t(x,y)\,\d\mu(x)\Big)\,\d\mu(y)\d t\\
&=&\int_0^\infty\int_M|\nabla P_tf|_w^2(y)\Big(\int_M p_t(y,x)\,\d\mu(x)\Big)\,\d\mu(y)\d t\\
&=&\int_M\int_0^\infty|\nabla P_tf|_w^2(y)\,\d t\d\mu(y);
\end{eqnarray*}
hence, $\|\mathcal{H}_\ast(f)\|_{L^2(M,\mu)}=\|\mathcal{H}(f)\|_{L^2(M,\mu)}\leq\|f\|_{L^2(M,\mu)}/2$.

Following \cite[Page 252]{PV2002} (see also \cite{BO2016}), we define the $p$-heat weight corresponding to the heat flow in the metric measure space.
\begin{definition}\label{heat-Ap}
Let $w: M\rightarrow [0,\infty]$ be a locally integrable function. For  $p\in(1,\infty)$, we say that $w$ is a $p$-heat weight, denoted by $w\in A_p^{heat}(M)$, if
$$\|w\|_{A_p^{heat}(M)}:=\big\|P_tw(P_tw^{-1/(p-1)})^{p-1}\big\|_{L^\infty(M\times[0,\infty),\mu\times \L^1)}<\infty,$$
where $\L^1$ is the one-dimensional Lebesgue measure restricted on $[0,\infty)$.
\end{definition}
Note that, by the H\"{o}lder inequality, we immediately have, for any $1<s\leq t<\infty$,
$$A_s^{heat}(M)\subset A_t^{heat}(M).$$
Indeed, for $w\in A_s^{heat}(M)$, H\"{o}lder's inequality implies that
\begin{eqnarray*}
&&\big(P_\tau w^{-1/(t-1)}\big)^{t-1}(x)=\Big(\int_M w(y)^{-1/(t-1)}p_\tau(x,y)\,\d\mu(y)\Big)^{t-1}\\
&\leq&\Big(\int_M (w(y)^{-1/(t-1)})^{(t-1)/(s-1)}p_\tau(x,y)\,\d\mu(y)\Big)^{s-1}\Big(\int_M p_\tau(x,y)\,\d\mu(y)\Big)^{t-s}\\
&\leq&\big(P_\tau w^{-1/(s-1)}\big)^{s-1}(x),
\end{eqnarray*}
for any $\tau\geq0$ and $x\in M$.

Now we are ready to present the main results. For any $p\in [1,\infty)$ and non-negative function $w\in L^1_{\loc}(M)$, let
$$L^p_w(M,\mu)=\Big\{f:M\rightarrow\R\mbox{ measurable }\Big| \int_M |f|^pw\,\d\mu<\infty\Big\},$$
and the norm of $f\in L^p_w(M,\mu)$ is defined by
$$\|f\|_{L^p_w(M,\mu)}=\Big(\int_M |f|^pw\,\d\mu\Big)^{1/p}.$$
\begin{theorem}\label{main}
Let $(M,d,\mu)$ be an $\RCD^\ast(0,N)$ space with $N\in[1,\infty)$ and $w\in A_2^{heat}(M)$.
Suppose that
\begin{equation}\label{assumption-main}
\limsup_{r\rightarrow\infty}\frac{\mu(B(o,r))}{r^N}>0,\quad\mbox{for some }o\in M.
\end{equation}
Then, for every $f\in C_c(M)$,
\begin{equation}\label{main-1}
\|f\|_{L^2_w(M,\mu)}\leq (320\|w\|_{A_2^{heat}(M)})^{1/2}\|\mathcal{H}_\ast(f)\|_{L^2_w(M,\mu)},
\end{equation}
\begin{equation}\label{main-2}
\|\mathcal{H}_\ast(f)\|_{L^2_w(M,\mu)}\leq \frac{\sqrt{2}}{2}\inf_{1<s<2}\Big(\frac{s}{2-s}\|w\|_{A_s^{heat}(M)}\Big)^{1/2}\|f\|_{L^2_w(M,\mu)},
\end{equation}
and moreover,
\begin{equation}\label{main-3}
\|\mathcal{H}_\ast(f)\|_{L^2_w(M,\mu)}\leq 2^{5/4}\|w\|_{A_2^{heat}(M)}\|f\|_{L^2_w(M,\mu)},
\end{equation}
\begin{equation}\label{main-4}
\|\mathcal{H}(f)\|_{L^2_w(M,\mu)}\leq 2^{7/4}\|w\|_{A_2^{heat}(M)}\|f\|_{L^2_w(M,\mu)}.
\end{equation}
\end{theorem}

An immediate observation is that $\limsup_{r\rightarrow\infty}\frac{\mu(B(x,r))}{r^N}$  is independent of $x$. Indeed, for any $x, y\in M$,
$$\limsup_{r\rightarrow\infty}\frac{\mu(B(x,r))}{r^N}\leq\limsup_{r\rightarrow\infty}\frac{\mu(B(y,r+d(x,y)))}{r^N}
=\limsup_{r\rightarrow\infty}\frac{\mu(B(y,r))}{r^N},$$
and the same inequality holds if we interchange the roles of $x$ and $y$. In addition, if the dimension $N$ is required to be an integer no less than 2, then sharper heat kernel estimates can be established; see \cite[Theorem 3.12]{LiH2017}.

Let $Z=\big((Z_t)_{t\geq0}, (\P^x)_{x\in M\setminus \mathcal{N}}\big)$ be the $\mu$-symmetric Hunt process corresponding to the Dirichlet form $(\D, W^{1,2}(M))$, where $\mathcal{N}$ is a properly exceptional set in the sense that $\mu(\mathcal{N})=0$ and $\P^x(Z_t\in \mathcal{N}\mbox{ for some }t>0)=0$ for all $x\in M\setminus\mathcal{N}$. Indeed, $Z$ is a $\mu$-symmetric diffusion with continuous path in the sense that
$$\P^x\big(t\mapsto Z_t\mbox{ is continuous for }t\in[0,\zeta)\big)=1,\quad\mbox{for every }x\in M\setminus\mathcal{N},$$
where $\zeta$ is the life time of $Z$. See  \cite{FOT2011} for instance.  Furthermore,  it can be shown  by the approach used to prove \cite[Thoerem 1.2 (c)]{ASZ2009} that
$$\P^x\big(t\mapsto Z_t\mbox{ is continuous for }t\in(0,\infty)\big)=1,\quad\mbox{for every }x\in M.$$

Now fix $T>0$. For $f\in C_c(M)$, define the processes $\mathcal{M}(f)=(\mathcal{M}(f)_t)_{0\leq t\leq T}$ and $\mathcal{N}(f)=(\mathcal{N}(f)_t)_{0\leq t\leq T}$  by
$$\mathcal{M}(f)_t=P_{T-t}f(Z_t)-P_Tf(Z_0),\quad 0\leq t\leq T,$$
and
$$\mathcal{N}(f)_t=P_{T-t}f(Z_t),\quad 0\leq t\leq T,$$
respectively. Denote the natural filtration of the process $(Z_t)_{t\geq0}$ by $(\mathcal{F}_t)_{t\geq0}$. Then the following lemma shows that $(\mathcal{M}(f)_t,\mathcal{F}_t)_{0\leq t\leq T}$ and $(\mathcal{N}(f)_t,\mathcal{F}_t)_{0\leq t\leq T}$ are martingales. We should mention that the result is not new in the smoothing setting and can be derived directly from It\^{o}'s formula; see e.g. \cite{BM2003}.  The approach of proof employed below is general and does not depend on It\^{o}'s formula (see a recent work \cite[Section 3]{LiWang2016} for non-local Dirichlet forms case).
\begin{lemma}\label{martingale}
Suppose that $(M,d,\mu)$ is an $\RCD^\ast(K,N)$ space with $K\in\R$ and $N\in[1,\infty)$. Let $T>0$ and $f\in C_c(M)$.  Then $(\mathcal{M}(f)_t,\mathcal{F}_t)_{0\leq t\leq T}$ and $(\mathcal{N}(f)_t,\mathcal{F}_t)_{0\leq t\leq T}$ defined above are uniformly integrable martingales with continuous path, and moreover, for any $t\in[0,T]$, the quadratic variations are
$$\langle\mathcal{M}(f)\rangle_t=2\int_0^t|\nabla P_{T-r}f|_w^2(Z_r)\,\d r,$$
and
$$\langle\mathcal{N}(f)\rangle_t=|P_Tf(x)|^2 + 2\int_0^t|\nabla P_{T-r}f|_w^2(Z_r)\,\d r,$$
respectively.
\end{lemma}
\begin{proof}
We only need to prove the assertions for $\mathcal{M}(f)$, since the proof for $\mathcal{N}(f)$ is similar. Let $0\leq s\leq t\leq T$. By the Markov property,
$$P_{T-t}f(Z_t)=\E_{Z_t}f(Z_{T-t})=\E[f(Z_T)|\mathcal{F}_t],$$
and hence
\begin{eqnarray*}
\E[\mathcal{M}(f)_t|\mathcal{F}_s]&=&\E[P_{T-t}f(Z_t)-P_Tf(Z_0)|\mathcal{F}_s]\\
&=&\E[P_{T-t}f(Z_t)|\mathcal{F}_s]-P_Tf(Z_0)\\
&=&\E\{\E[f(Z_T)|\mathcal{F}_t]|\mathcal{F}_s\}-P_Tf(Z_0)\\
&=&\E[f(Z_T)|\mathcal{F}_s]-P_Tf(Z_0)\\
&=&P_{T-s}f(Z_s)-P_Tf(Z_0)\\
&=&\mathcal{M}(f)_s,
\end{eqnarray*}
which implies that $(\mathcal{M}(f)_t,\mathcal{F}_t)_{0\leq t\leq T}$ is a martingale with continuous path, since $t\mapsto Z_t$ is continuous and the map $(t,x)\mapsto P_tf(x)$ belongs to $C_b((0,\infty)\times M)$ (see \cite[Theorem 7.1 (iii)]{agmr2015}). In addition, it is easy to know that the family  $\{\mathcal{M}(f)_t:0\leq t\leq T\}$ is uniformly integrable

For every $x\in M$, since
\begin{eqnarray*}
\E_x[\mathcal{M}(f)_t^2]&=&\E_x[(P_{T-t}f(Z_t)-P_Tf(Z_0))^2]\\
&=&\E_x[(P_{T-t}f(Z_t))^2]-2P_Tf(x)\E_x[P_{T-t}f(Z_t)]+(P_Tf(x))^2\\
&=&P_t(P_{T-t}f)^2(x)-(P_Tf(x))^2,
\end{eqnarray*}
we have
\begin{eqnarray*}
\E_x[\mathcal{M}(f)_t^2-\mathcal{M}(f)_s^2]&=&P_t(P_{T-t}f)^2(x)-P_s(P_{T-s}f)^2(x)\\
&=&\int_s^t \frac{\d P_r(P_{T-r}f)^2(x)}{\d r}\,\d r\\
&=&\int_s^t\big(\Delta P_r(P_{T-r}f)^2(x)-2P_r\big(P_{T-r}f\cdot \Delta (P_{T-r}f)\big)(x)\big)\,\d r\\
&=&\int_s^tP_r\big(\Delta (P_{T-r}f)^2 -2P_{T-r}f\cdot \Delta(P_{T-r}f)\big)(x)\,\d r\\
&=&2\int_s^tP_r\big(\Gamma(P_{T-r}f)\big)(x)\,\d r\\
&=&\E_x\Big[2\int_s^t|\nabla P_{T-r}f|_w^2(Z_r)\,\d r\Big],
\end{eqnarray*}
where in the last line  we used the fact that (see e.g. \cite{AmbrosioGigliSavare2011b,AMS2016})
$$|\nabla f|_w^2=\Gamma(f,f):=\frac{1}{2}\big(\Delta(f^2)-2f\Delta f\big),\quad\mbox{for every }f\in W^{1,2}(M).$$
Hence, applying the Markov property again, we derive that
$$\E_x\Big[\mathcal{M}(f)_t^2-2\int_0^t|\nabla P_{T-r}f|_w^2(Z_r)\,\d r\Big|\mathcal{F}_s\Big]=\mathcal{M}(f)_s^2-2\int_0^s|\nabla P_{T-r}f|_w^2(Z_r)\,\d r.$$
Thus,
$$\Big(\mathcal{M}(f)_t^2-2\int_0^t|\nabla P_{T-r}f|_w^2(Z_r)\,\d r,\, \mathcal{F}_t\Big)_{0\leq t\leq T}$$
is also a martingale.

Therefore, since $\mathcal{M}(f)_0=0$, the quadratic variation of $\mathcal{M}(f)_t$ is
\begin{eqnarray*}
\langle\mathcal{M}(f)\rangle_t=2\int_0^t|\nabla P_{T-r}f|_w^2(Z_r)\,\d r.
\end{eqnarray*}\end{proof}

The next lemma  expresses the $\mathcal{H}_{\ast,T}$-function as a conditional distribution of the quadratic variation of $\mathcal{M}(f)$, given $Z_T$. The proof, which we present here for the sake of completeness, is the same as the one for \cite[(5.18)]{BO2016}. Note that only the symmetry of the heat kernel and the stochastic completeness are used in the proof.
\begin{lemma}\label{H_T}
Suppose that $(M,d,\mu)$ is an $\RCD^\ast(K,N)$ space with $K\in\R$ and $N\in[1,\infty)$. Let $T>0$, $f\in C_c(M)$  and $x\in M$. Define
$$\mathcal{H}_{\ast,T}(f)(x)=\Big(\int_0^T\int_M|\nabla P_tf|_w^2(y)p_t(x,y)\,\d\mu(y)\d t\Big)^{1/2}.$$
Then $\lim_{T\rightarrow\infty}\mathcal{H}_{\ast,T}(f)(x)=\mathcal{H}_{\ast}(f)(x)$, and
$$\mathcal{H}_{\ast,T}(f)(x)=\Big(\int_M \mathbb{E}_y\Big[\int_0^T|\nabla P_{T-r}f|_w^2(Z_r)\,\d r\Big| Z_T=x\Big]p_T(x,y)\,\d\mu(y)\Big)^{1/2}.$$
\end{lemma}
\begin{proof} We only need to prove the second assertion. By the symmetry of the heat kernel and the stochastic completeness,
\begin{eqnarray*}
\mathcal{H}_{\ast,T}(f)^2(x)&=&\int_0^T\int_M|\nabla P_{s}f|_w^2(z)p_{s}(x,z)\,\d\mu(z)\d s\\
&=&\int_0^T\int_Mp_{T-r}(z,x)|\nabla P_{T-r}f|_w^2(z)\,\d\mu(z)\d r\\
&=&\int_0^T\int_Mp_{T-r}(z,x)|\nabla P_{T-r}f|_w^2(z)\Big(\int_M p_r(y,z)\,\d\mu(y)\Big)\,\d\mu(z)\d r\\
&=&\int_M\Big(\int_0^T\int_M\frac{p_r(y,z)p_{T-r}(z,x)}{p_T(y,x)}|\nabla P_{T-r}f|_w^2(z)\,\d\mu(z)\d r\Big)p_T(y,x)\,\d\mu(y)\\
&=&\int_M \mathbb{E}_y\Big[\int_0^T|\nabla P_{T-r}f|_w^2(Z_r)\,\d r\Big| Z_T=x\Big]p_T(x,y)\,\d\mu(y),
\end{eqnarray*}
where we used the definition of the conditional distribution of $Z_r$ under $\P^y$ given $Z_T=x$ in the last equality.
\end{proof}

We borrow a lemma from \cite[Lemma 5.2]{BO2016} and omit its proof here.
\begin{lemma}\label{muck-heat}
Fix $T>0$ and $x\in M$. Let $(Z_t)_{t\geq0}$ be the diffusion process as above with $Z_0=x$. Suppose $w\in A_p^{heat}(M)$. Consider the process $Y_t=P_{T-t}w(Z_t)$, $0\leq t\leq T$, under the probability measure $\P^x$.  Then $Y_T\in A_p^{mart}$ and $$\|Y_T\|_{A_p^{mart}}\leq\|w\|_{A_p^{heat}(M)}.$$
\end{lemma}

Now we are ready to prove Theorem \ref{main}. The basic idea of proof is not new. It comes from \cite[Section 5]{BO2016} for Cauchy semigroups in the Euclidean setting.
\begin{proof}[Proof of Theorem \ref{main}]
(1) Fix $T>0$. Let $f\in C_c(M)$ and $K=\supp(f)$. Then,  by the stochastic completeness and the symmetry of the heat kernel, we have
\begin{eqnarray*}
\int_M |f(x)|^2w(x)\,d\mu(x)&=&\int_M\int_M |f(y)|^2w(y)p_T(y,x)\,\d\mu(x)\d\mu(y)\\
&=&\int_M\int_M \mathbf{1}_K(y)|f(y)|^2w(y)p_T(y,x)\,\d\mu(x)\d\mu(y)\\
&=&\int_M\Big(\int_M \mathbf{1}_K(y)|f(y)|^2w(y)p_T(x,y)\,\d\mu(y)\Big)\d\mu(x)\\
&=&\int_M \mathbb{E}_x\Big[\mathbf{1}_K(Z_T)|f(Z_T)|^2w(Z_T)\Big]\,\d\mu(x)\\
&\leq&2\int_M \mathbb{E}_x\Big[\mathbf{1}_K(Z_T)|f(Z_T)-P_Tf(x)|^2w(Z_T)\Big]\,\d\mu(x)\\
  &&+  2\int_M \mathbb{E}_x\Big[\mathbf{1}_K(Z_T)|P_Tf(x)|^2w(Z_T)\Big]\,\d\mu(x)\\
&=&: {\rm I}+ {\rm II}.
\end{eqnarray*}
Let $Y_t=P_{T-t}w(Z_t)$ for $0\leq t\leq T$. Then $Y_T=w(Z_T)$. Define $\d\mathbb{Q}=Y_T\d\P$. Applying \eqref{L2-1} and Lemma \ref{martingale}, we derive that
\begin{eqnarray*}
&&\mathbb{E}_x\big[|f(Z_T)-P_Tf(x)|^2w(Z_T)\big]\\
&=&\mathbb{E}_x\big[\mathcal{M}(f)_T^2Y_T\big]=\|\mathcal{M}(f)_T\|^2_{L^2(\mathbb{Q})}\\
&\leq&80\|Y_T\|_{A_2^{mart}}\|\langle\mathcal{M}(f)\rangle^{1/2}_T\|^2_{L^2(\mathbb{Q})}\\
&=&160\|Y_T\|_{A_2^{mart}}\mathbb{E}_x\Big[w(Z_T)\int_0^T|\nabla P_{T-r}f|^2_w(Z_r)\,\d r \Big],
\end{eqnarray*}
Then, by the Markov property, we  obtain  that
\begin{eqnarray*}
{\rm I}&:=&\int_M\mathbb{E}_x\Big[w(Z_T)\int_0^T|\nabla P_{T-r}f|^2_w(Z_r)\,\d r \Big]\,\d\mu(x)\\
&=&\int_M\mathbb{E}_x\Big\{\mathbb{E}_x\Big[w(Z_T)\int_0^T|\nabla P_{T-r}f|^2_w(Z_r)\,\d r\Big|\mathcal{F}_T\Big]\Big\}\,\d\mu(x)\\
&=&\int_M\mathbb{E}_x\Big\{\mathbb{E}_x\Big[\int_0^T|\nabla P_{T-r}f|^2_w(Z_r)\,\d r\Big|\mathcal{F}_T\Big] w(Z_T)\Big\}\,\d\mu(x)\\
&=&\int_M\mathbb{E}_x\Big\{\mathbb{E}_{x}\Big[\int_0^T|\nabla P_{T-r}f|^2_w(Z_r)\,\d r\Big|Z_T\Big] w(Z_T)\Big\}\,\d\mu(x).
\end{eqnarray*}
Denote $$\varphi(Z_T)=\mathbb{E}_{x}\Big[\int_0^T|\nabla P_{T-r}f|^2_w(Z_r)\,\d r\Big|Z_T\Big].$$ Then
\begin{eqnarray*}
{\rm I}&=&\int_M\mathbb{E}_x\left[\varphi(Z_T) w(Z_T)\right]\,\d\mu(x)\\
&=&\int_M\Big(\int_M\varphi(y)w(y)p_T(x,y)\,\d\mu(y)\Big)\,\d\mu(x)\\
&=&\int_M\Big(\int_M\varphi(y)p_T(x,y)\,\d\mu(x)\Big)w(y)\,\d\mu(y)\\
&=&\int_M\Big(\int_M\mathbb{E}_{x}\Big[\int_0^T|\nabla P_{T-r}f|^2_w(Z_r)\,d r\Big|Z_T=y\Big]p_T(x,y)\,\d\mu(x)\Big)w(y)\,\d\mu(y)\\
&=&\int_M\mathcal{H}_{\ast,T}(f)(y)^2w(y)\,\d\mu(y),
\end{eqnarray*}
where we applied Lemma \ref{H_T} in the last line. Thus,
\begin{eqnarray*}
{\rm I}&\leq& 320\|Y_T\|_{A_2^{mart}}\int_M\mathcal{H}_{\ast,T}(f)(x)^2  w(x)\,\d\mu(x)\\
&\leq&320\|w\|_{A_2^{heat}(M)}\int_M\mathcal{H}_{\ast,T}(f)(x)^2  w(x)\,\d\mu(x),
\end{eqnarray*}
where we applied Lemma \ref{muck-heat} in the second inequality.

Applying the heat kernel upper bound \eqref{gassian-bound-1}, we immediately have
\begin{eqnarray*}
|P_Tf(x)|^2&\leq& P_T(|f|^2)(x)=\int_M p_T(x,y)|f(y)|^2\,\d\mu(y)\\
&\leq&\int_M \frac{C}{\mu(B(x,\sqrt{T}))}|f(y)|^2\,\d\mu(y)\\
&=&\frac{C}{\mu(B(x,\sqrt{T}))}\|f\|^2_{L^2(M,\mu)},
\end{eqnarray*}
where $C$ is a positive constant. By the assumption \eqref{assumption-main}, we have
\begin{eqnarray}\label{lower-vol}
\mu(B(x,r))\geq \theta r^N,\quad\mbox{for any }x\in M,\, r>0,
\end{eqnarray}
where $\theta:=\limsup_{r\rightarrow\infty}[\mu(B(o,r))/r^N]>0$. Indeed, by the volume comparison property in Proposition \ref{vol-comparison}, for any $x\in M$ and $r>0$,
\begin{eqnarray*}
\frac{\mu(B(x,r))}{r^N}&\geq&\limsup_{R\rightarrow\infty}\frac{\mu(B(x,R))}{R^N}\\
&\geq&\limsup_{R\rightarrow\infty}\frac{\mu(B(o,R-d(o,x)))}{R^N}\\
&=&\limsup_{R\rightarrow\infty}\Big(\frac{\mu(B(o,R-d(o,x)))}{(R-d(o,x))^N}\cdot\frac{(R-d(o,x))^N}{R^N}\Big)\\
&\geq&\theta.
\end{eqnarray*}
Hence,
\begin{eqnarray*}
{\rm II}&=&2\int_M |P_Tf(x)|^2\mathbb{E}_x[\mathbf{1}_K(Z_T)w(Z_T)]\,\d\mu(x)\\
&\leq&C\|f\|^2_{L^2(M,\mu)}\int_M\frac{1}{\mu(B(x,\sqrt{T}))}\Big(\int_Kw(y)p_T(x,y)\,\d\mu(y)\Big)\d\mu(x)\\
&\leq&\frac{C}{\theta}\|f\|^2_{L^2(M,\mu)}T^{-N/2}\int_M\int_K w(y)p_T(x,y)\,\d\mu(y)\d\mu(x)\\
&=&\frac{C}{\theta}\|f\|^2_{L^2(M,\mu)}T^{-N/2}\int_K w(y)\,\d\mu(y),
\end{eqnarray*}
where we used the symmetry of the heat kernel and the stochastic completeness again in the last line. Obviously, the term in the last line tends to 0 as $T\rightarrow\infty$.

Therefore, combining the estimates on ${\rm I}$ and ${\rm II}$, we arrive at
\begin{eqnarray*}
\int_M |f(x)|^2w(x)\,d\mu(x)&\leq&320\|w\|_{A_2^{heat}(M)}\int_M\mathcal{H}_{\ast,T}(f)(x)^2  w(x)\,\d\mu(x)\\
&&+\frac{C}{\theta}\|f\|^2_{L^2(M,\mu)}T^{-N/2}\int_K w(y)\,\d\mu(y).
\end{eqnarray*}
Letting $T\rightarrow\infty$, by the monotone convergence theorem, we prove \eqref{main-1}.

(2) Similar as the argument above, we have
\begin{eqnarray*}
&&\int_M \mathcal{H}_{\ast,T}(f)^2(y)w(y)\,d\mu(y)\\
&=&\int_M\mathbb{E}_x\Big[w(Z_T)\int_0^T|\nabla P_{T-r}f|^2_w(Z_r)\,\d r\Big]\,\d\mu(x)\\
&\leq&\frac{1}{2}\int_M \E_x\Big[\Big(|P_Tf(x)|^2+2\int_0^T|\nabla P_{T-r}f|^2_w(Z_r)\,\d r\Big)w(Z_T)\Big]\,\d\mu(x)\\
&=&\frac{1}{2}\int_M \E_x\big[\langle\mathcal{N}(f)\rangle_Tw(Z_T)\big]\,\d\mu(x).
\end{eqnarray*}
Then, applying \eqref{L2-2}, we obtain
\begin{eqnarray*}
&&\int_M \mathcal{H}_{\ast,T}(f)^2(y)w(y)\,\d\mu(y)\\
&\leq& \frac{1}{2}\int_M \inf_{1<s<2}\Big(\frac{s}{2-s}\|Y_T\|_{A_s^{mart}}\Big)\E_x\big[\mathcal{N}(f)_T^2w(Z_T)\big]\,\d\mu(x)\\
&=&\frac{1}{2}\int_M \inf_{1<s<2}\Big(\frac{s}{2-s}\|Y_T\|_{A_s^{mart}}\Big)\E_x\big[|f(Z_T)|^2w(Z_T)\big]\,\d\mu(x)\\
&\leq&\frac{1}{2}\inf_{1<s<2}\Big(\frac{s}{2-s}\|w\|_{A_s^{heat}(M)}\Big) \int_M |f(x)|^2w(x)\,\d\mu(x),
\end{eqnarray*}
where we used Lemma \ref{muck-heat} in the last inequality.  Thus, we prove \eqref{main-2}.

(3) By the same approach as above, applying \eqref{L2-3} instead of \eqref{L2-2}, we obtain \eqref{main-3}.

(4) By the inequality (see e.g. \cite{AmbrosioGigliSavare2012})
$$|\nabla P_tf|_w^2\leq P_t|\nabla f|_w^2,\quad\mbox{for every }f\in W^{1,2}(M),$$
we deduce that
\begin{eqnarray*}
\mathcal{H}(f)^2(x)&=&\int_0^\infty |\nabla P_tf|_w^2(x)\,\d t\\
&\leq&\int_0^\infty P_{t/2}\big(|\nabla P_{t/2}f|_w^2\big)(x)\,\d t\\
&=&\int_0^\infty\int_M |\nabla P_{t/2}f|^2_w(y)p_{t/2}(x,y)\,\d\mu(y)\d t\\
&=&2\int_0^\infty\int_M |\nabla P_{t}f|^2_w(y)p_{t}(x,y)\,\d\mu(y)\d t\\
&=&2\mathcal{H}_\ast(f)^2(x).
\end{eqnarray*}
Thus, combining this and \eqref{main-3}, we obtain \eqref{main-4} immediately.

Therefore, the proof is completed.
\end{proof}
\begin{remark}\label{remark-of-main} The $\RCD^*(0,N)$ space with $N\in [1,\infty)$ turns out to be a convenient setting to establish Theorem \ref{main}. However, it seems that the method also works in more general situations.

(1) In the above proof, we do not use the full upper bound of the heat kernel in \eqref{gassian-bound-1}; indeed, the estimate
$$p_t(x,y)\leq \frac{c}{\mu\big(B(x,\sqrt{t})\big)}{},\quad\mbox{ for any }x,y\in M\mbox{ and }t>0,$$
 is enough, where $c$ is a positive constant.

(2) In fact, for Theorem \ref{main} to hold, it is enough to assume that, there exist some constants $C>0$ and $\delta>0$ such that
\begin{eqnarray}\label{lower-vol-growth}\mu\big(B(x_0,r)\big)\geq Cr^\delta,\quad\mbox{ for some }x_0\in M\mbox{ and for any big }r>0,\end{eqnarray}
in stead of the maximum volume growth assumption \eqref{assumption-main}. It is clear that \eqref{lower-vol-growth} does not depend on the choice of $x_0$, since from Proposition \ref{vol-comparison}, for any $x\in M$,
$$\mu\big(B(x_0,r)\big)\leq\mu\big(B(x,d(x_0,x)+r)\big)\leq\Big(\frac{r+d(x_0,x)}{r}\Big)^N\mu\big(B(x,r)\big)\leq2^N\mu\big(B(x,r)\big),$$
for any $r\geq d(x_0,x)$. Moreover, every noncompact $\CD^\ast(0,N)$ space (in particular, $\RCD^\ast(0,N)$ space) $(M,d,\mu)$ has at least linear growth of the $\mu$-measure of the ball, which is implied by Lemma \ref{low-vol-bound} below.
\end{remark}

\begin{lemma}\label{low-vol-bound}
Let $(M,d,\mu)$ be a noncompact $\CD^\ast(0,N)$ space with $N\in[1,\infty)$. Then $\mu(M)=\infty$, and moreover, for every $o\in M$, there exists a positive constant $C$ such that
\begin{eqnarray}\label{low-vol-bound-1}
\liminf_{r\rightarrow\infty}\frac{\mu\big(B(o,r)\big)}{r}>C.
\end{eqnarray}
\end{lemma}
\begin{proof}
Fix an arbitrary $r>0$. For any $t>r$ and any $y\in M$, by the volume comparison property in Proposition \ref{vol-comparison},
\begin{eqnarray*}
&&\frac{\mu\big(B(y,t)\big)}{t^N}-\frac{\mu\big(B(y,t)\big)-\mu\big(B(y,r)\big)}{t^N-r^N}\\
&=&\frac{r^N}{t^N-r^N}\Big[\frac{\mu\big(B(y,r)\big)}{r^N}-\frac{\mu\big(B(y,t)\big)}{t^N}\Big]\geq0.
\end{eqnarray*}
Let $x\in \partial B(o,t)$ (the boundary of the ball $B(o,t)$). The triangular inequality implies that $B(o,r)\subset B(x,t+r)\setminus B(x,t-r)$ and $B(x,t-r)\subset B(o,2t)$. Hence,
\begin{eqnarray*}
&&\mu\big(B(o,2t)\big)\geq\mu\big(B(x,t-r)\big)\\
&\geq& \frac{(t-r)^N}{(t+r)^N-(t-r)^N}\Big[\mu\big(B(x,t+r)\big) - \mu\big(B(x,t-r)\big)\Big]\\
&\geq&\frac{(t-r)^N}{(t+r)^N-(t-r)^N}\,\mu\big(B(o,r)\big).
\end{eqnarray*}
Let $\tilde{N}=\lfloor N\rfloor+1$, where $\lfloor N\rfloor$ is the integer part of $N$. For any $t\geq 2r$, since
\begin{eqnarray*}
&&\frac{(t-r)^N}{(t+r)^N-(t-r)^N}=\frac{1}{\big(1+\frac{2r}{t-r}\big)^N-1}\geq \frac{1}{\big(1+\frac{2r}{t-r}\big)^{\tilde{N}}-1}\\
&=&\frac{1}{\sum\limits_{k=1}^{\tilde{N}}\tbinom{\tilde{N}}{k}\big(\frac{2r}{t-r}\big)^k}\geq \frac{t-r}{r}\cdot\frac{1}{\sum\limits_{k=1}^{\tilde{N}}\tbinom{\tilde{N}}{k}2^k}
\geq\frac{t-t/2}{r}\cdot\frac{1}{\sum\limits_{k=1}^{\tilde{N}}\tbinom{\tilde{N}}{k}2^k}\\
&=&\frac{t}{2r\sum\limits_{k=1}^{\tilde{N}}\tbinom{\tilde{N}}{k}2^k}=\frac{t}{2(3^{\tilde{N}}-1)r}=ct,
\end{eqnarray*}
where $\frac{1}{c}=2(3^{\tilde{N}}-1)r$. Thus,
$$\mu\big(B(o,2t)\big)\geq c\mu\big(B(o,r)\big)t,\quad\mbox{for any }t\geq 2r.$$
It combined with Remark \ref{remark-of-main} (2) implies \eqref{low-vol-bound-1}, and hence $\mu(M)=\infty$ obviously. The proof is completed.
\end{proof}
\begin{remark}\label{remark-of-lemma}
Lemma \ref{low-vol-bound} is a generalization of the result obtained separately by E. Calabi in \cite{Calabi1975}  and S.T. Yau  in \cite{Yau1975} on complete and noncompact Riemannian manifolds with nonnegative Ricci curvature. The proof of Lemma \ref{low-vol-bound} is elementary and  only depends on the volume comparison property (see Proposition \ref{vol-comparison} above). Hence, it seems that property \eqref{low-vol-bound-1} also holds in the more general setting, i.e., the noncompact metric measure space satisfying the so-called measure contraction property $\MCP(0,N)$ with $N\in [1,\infty)$ (see e.g. \cite[Definition 2.1 and Theorem 5.1]{Ohta2007} for the definition and the volume comparison property, as well as \cite[Section 5]{Sturm2006b}).
\end{remark}

From Remark \ref{remark-of-main} (2) and Lemma \ref{low-vol-bound}, we immediately obtain the following proposition.
\begin{proposition}\label{noncom-rcd-main}
Let $(M,d,\mu)$ be a noncompact $\RCD^\ast(0,N)$ space with $N\in[1,\infty)$. Then all the inequalities \eqref{main-1}--\eqref{main-4} hold.
\end{proposition}

\medskip

Now we turn to  study the  Lusin area function in the metric measure space setting. Let $\alpha>0$ and $x\in M$. Define
$$\mathcal{C}_\alpha(x)=\{(y,t)\in M\times\R_+: d(y,x)<\alpha\sqrt{t}\},$$
which is the so-called parabolic cone with vertex at $x$ and aperture $\alpha$.
For every $f\in C_c(M)$, define the Lusin area function as
$$\mathcal{A}_\alpha(f)(x)=\Big(\int_{\mathcal{C}_\alpha(x)}t^{-N/2}|\nabla P_tf|_w^2(y)\,\d\mu(y)\d t\Big)^{1/2}.$$
\begin{corollary}\label{weighted-area}
Let $(M,d,\mu)$ be an $\RCD^\ast(0,N)$ space with $N\in[1,\infty)$ and $w\in A_2^{heat}(M)$.
Suppose that \eqref{assumption-main} (or \eqref{lower-vol-growth}) holds  and
\begin{equation}\label{assumption-area}
\liminf_{r\rightarrow0}\frac{\mu(B(x,r))}{r^N}=\kappa,
\end{equation}
for every $x\in M$ and  some constant $\kappa>0$. Then, for every $f\in C_c(M)$,
$$\|\mathcal{A}_\alpha(f)\|_{L^2_w(M,\mu)}\leq \Big(\frac{\kappa}{C}\Big)^{1/2}e^{\alpha^2/6}\|w\|_{A_2^{heat}(M)}\|f\|_{L^2_w(M,\mu)},$$
for some constant $C>0$ depending on $N$.
\end{corollary}
\begin{proof}
From Proposition \ref{vol-comparison}, we know that the function $r\mapsto\mu(B(x,r))/r^N$ is non-decreasing in $(0,\infty)$ for every $x\in M$.
Hence, the assumption \eqref{assumption-area} implies that
$$\frac{\mu\big(B(x,R)\big)}{R^N}\leq\liminf_{r\rightarrow0}\frac{\mu\big(B(x,r)\big)}{r^N}= \kappa,$$
for any $x\in M$ and all $R>0$.
Applying the heat kernel lower bound in \eqref{gassian-bound-1}, we derive
\begin{eqnarray*}
p_t(x,y)&\geq& \frac{C}{\mu(B(x,\sqrt{t}))}\exp\Big(-\frac{d^2(x,y)}{3t}\Big)\\
&\geq&\frac{C}{\kappa t^{N/2}}\exp\Big(-\frac{d^2(x,y)}{3t}\Big),
\end{eqnarray*}
for any $y\in M$. Then
$$t^{-N/2}\leq \frac{\kappa}{C}\exp\Big(\frac{d^2(x,y)}{3t}\Big)p_t(x,y).$$
Hence,
\begin{eqnarray*}
\mathcal{A}_\alpha(f)(x)&\leq&\Big(\int_{\mathcal{C}_\alpha(x)}\frac{\kappa}{C}e^{\alpha^2/3}|\nabla P_tf|_w^2(y)p_t(x,y)\,\d\mu(y)\d t\Big)^{1/2}\\
&\leq&\Big(\frac{\kappa}{C}\Big)^{1/2}e^{\alpha^2/6}\mathcal{H}_\ast(f)(x).
\end{eqnarray*}
Thus, combining this with \eqref{main-3}, we  complete the proof.
\end{proof}
Clearly, in the noncompact setting, the inequality in Corollary \ref{weighted-area} also holds without assumption \eqref{assumption-main} (or \eqref{lower-vol-growth}), due to Proposition \ref{noncom-rcd-main}.

\section{The comparison of $p$-heat weight and $p$-Muckenhoupt weight}
Now let us recall the definition of the $p$-Muckenhoupt weight in the setting of metric measure spaces.
\begin{definition}\label{Muckenhoupt-weight}
Let $(M,d,\mu)$ be a metric measure space and $w: M\rightarrow [0,\infty]$ be a locally integrable function. For $p\in(1,\infty)$, we say that $w$ is a $p$-Muckenhoupt weight, denoted by $w\in A_p(M)$, if
$$\|w\|_{A_p(M)}:=\sup_B\Big(\frac{1}{\mu(B)}\int_Bw\,\d\mu\Big)\Big(\frac{1}{\mu(B)}\int_Bw^{-1/(p-1)}\,\d\mu\Big)^{p-1}<\infty,$$
where the supremum is taken over all balls $B\subset M$.
\end{definition}

It is an interesting question to ask about  the relationship between the $p$-heat  weight and  the $p$-Muckenhoupt weight in our metric measure space  setting.  The next theorem shows that  the $p$-heat  weight and the $p$-Muckenhoupt weight are comparable for each $1<p<\infty$. So, it is natural to regard the $p$-heat weight as a probabilistic representative of the $p$-Muckenhoupt weight. The same conclusion \eqref{comparison-of-weights-0} below for $p=2$ in $\R^2$ was obtained by Petermichl and Volberg in \cite[Theorem 3.1]{PV2002}.
\begin{theorem}\label{comparison-of-weights}
Let $(M,d,\mu)$ be an $\RCD^\ast(0,N)$ space with $N\in[1,\infty)$ and let $1<p<\infty$. Then, there exist positive constants $c_1$ and $c_2$ depending on $N$ such that
\begin{equation}\label{comparison-of-weights-0}
c_1\|w\|_{A_p^{heat}(M)}\leq\|w\|_{A_p(M)}\leq c_2\|w\|_{A_p^{heat}(M)}.
\end{equation}
\end{theorem}
\begin{proof}The constant $C$ below may vary from line to line. We assume that $w\in A_p(M)$ and then prove the first inequality in \eqref{comparison-of-weights-0}. For any $x\in M$ and $t>0$, let $B_{-1}=\emptyset$, $B_k=B(x,2^k\sqrt{t})$ and $C_k=B_k\setminus B_{k-1}$, $k=0,1,2,\cdots$. Then, by the heat kernel upper bound in \eqref{gassian-bound-1},
\begin{eqnarray*}
&&\Big(\frac{1}{\mu(B_0)}\int_{B_0}w\,\d\mu\Big)\big(P_t(w^{-1/(p-1)})(x)\big)^{p-1}\\
&=&\Big(\frac{1}{\mu(B_0)}\int_{B_0}w\,\d\mu\Big)\Big(\int_Mp_t(x,y)w^{-1/(p-1)}(y)\,\d\mu(y)\Big)^{p-1}\\
&\leq&\Big(\frac{1}{\mu(B_0)}\int_{B_0}w\,\d\mu\Big)\Big(\int_M\frac{C}{\mu(B_0)}\exp\Big[-\frac{d^2(x,y)}{5t}\Big]w^{-1/(p-1)}(y)\,\d\mu(y)\Big)^{p-1}\\
&\leq&\Big(\frac{1}{\mu(B_0)}\int_{B_0}w\,\d\mu\Big)\Big(\sum_{k=0}^\infty\int_{C_k}\frac{C}{\mu(B_0)}\exp\Big[-\frac{d^2(x,y)}{5t}\Big]w^{-1/(p-1)}(y)\,\d\mu(y)\Big)^{p-1}\\
&\leq&\Big(\frac{1}{\mu(B_0)}\int_{B_0}w\,\d\mu\Big)\Big(\frac{C}{\mu(B_0)}\sum_{k=0}^\infty\mu(B_k)\exp\Big(-\frac{2^{2k-2}}{5}\Big)\frac{1}{\mu(B_k)}
\int_{B_k}w^{-1/(p-1)}\,\d\mu\Big)^{p-1},
\end{eqnarray*}
where $C$ is a positive constant depending on $N$. Since
$$\Big(\frac{1}{\mu(B_k)}\int_{B_k}w\,\d\mu\Big)\Big(\frac{1}{\mu(B_k)}\int_{B_k}w^{-1/(p-1)}\,\d\mu\Big)^{p-1}\leq C,$$
for some constant $C>0$, we have
\begin{eqnarray*}
&&\Big(\frac{1}{\mu(B_0)}\int_{B_0}w\,\d\mu\Big)\big(P_t(w^{-1/(p-1)})(x)\big)^{p-1}\\
&\leq&\Big(\frac{1}{\mu(B_0)}\int_{B_0}w\,\d\mu\Big)\Big\{\frac{C}{\mu(B_0)}\sum_{k=0}^\infty\mu(B_k)\exp\Big(-\frac{2^{2k-2}}{5}\Big)\\
&&\times\Big(\frac{1}{\mu(B_k)}\int_{B_k}w\,\d\mu\Big)^{-1/(p-1)}\Big\}^{p-1},
\end{eqnarray*}
where, by the volume comparison property (see Proposition \ref{vol-comparison}),
\begin{eqnarray*}
\frac{1}{\mu(B_k)}\int_{B_k}w\,\d\mu&=&\frac{\mu(B_0)}{\mu(B_k)}\cdot\frac{1}{\mu(B_0)}\int_{B_k}w\,\d\mu\\
&\geq&2^{-kN}\frac{1}{\mu(B_0)}\int_{B_0}w\,\d\mu.
\end{eqnarray*}
Hence, by the volume comparison property in Proposition \ref{vol-comparison} again,
\begin{eqnarray*}
&&\Big(\frac{1}{\mu(B_0)}\int_{B_0}w\,\d\mu\Big)\big(P_t(w^{-1/(p-1)})(x)\big)^{p-1}\\
&\leq&\Big(\frac{1}{\mu(B_0)}\int_{B_0}w\,\d\mu\Big)\Big\{\frac{C}{\mu(B_0)}\sum_{k=0}^\infty\mu(B_k)\exp\Big(-\frac{2^{2k-2}}{5}\Big)\\
&&\times\Big(\frac{1}{\mu(B_0)}\int_{B_0}w\,\d\mu\Big)^{-1/(p-1)}2^{kN/(p-1)} \Big\}^{p-1}\\
&\leq&\Big\{\frac{C}{\mu(B_0)}\sum_{k=0}^\infty 2^{kN/(p-1)}\mu(B_k)\exp\Big(-\frac{2^{2k-2}}{5}\Big)\Big\}^{p-1}\\
&\leq&C\Big\{\sum_{k=0}^\infty2^{pkN/(p-1)}\exp\Big(-\frac{2^{2k-2}}{5}\Big)\Big\}^{p-1},
\end{eqnarray*}
which implies that
\begin{eqnarray}\label{comparison-of-weights-1}
\Big(\frac{1}{\mu(B_0)}\int_{B_0}w\,\d\mu\Big)\big(P_t(w^{-1/(p-1)})(x)\big)^{p-1}\leq C,
\end{eqnarray}
for any $x\in M$ and $t>0$, where $C$ is a positive constant depending on $N$. Then, applying the heat kernel upper bound in \eqref{gassian-bound-1} again, we derive that
\begin{eqnarray*}
&&P_tw(x)\big(P_t(w^{-1/(p-1)})(x)\big)^{p-1}\\
&=&\big(P_t(w^{-1/(p-1)})(x)\big)^{p-1}\int_Mp_t(x,y)w(y)\,\d\mu(y)\\
&\leq&\big(P_t(w^{-1/(p-1)})(x)\big)^{p-1}\sum_{k=0}^\infty\int_{C_k}p_t(x,y)w(y)\,\d\mu(y)\\
&\leq&\big(P_t(w^{-1/(p-1)})(x)\big)^{p-1}\sum_{k=0}^\infty\int_{B_k}\frac{C}{\mu(B_0)}\exp\Big(-\frac{2^{2k-2}}{5}\Big)w(y)\,\d\mu(y)\\
&=&\big(P_t(w^{-1/(p-1)})(x)\big)^{p-1}\sum_{k=0}^\infty\frac{C\mu(B_k)}{\mu(B_0)}\exp\Big(-\frac{2^{2k-2}}{5}\Big)
\Big(\frac{1}{\mu(B_k)}\int_{B_k}w(y)\,\d\mu(y)\Big)\\
&\leq&C\big(P_t(w^{-1/(p-1)})(x)\big)^{p-1}\sum_{k=0}^\infty 2^{kN}\exp\Big(-\frac{2^{2k-2}}{5}\Big)\big(P_{2^{2k}t}(w^{-1/(p-1)})(x)\big)^{-(p-1)},
\end{eqnarray*}
where we used the volume comparison property and \eqref{comparison-of-weights-1}. By the heat kernel lower and upper bounds \eqref{gassian-bound-1} and the volume comparison property,
\begin{eqnarray*}
P_{2^{2k}t}(w^{-1/(p-1)})(x)&=&\int_Mp_{2^{2k}t}(x,y)w^{-1/(p-1)}(y)\,\d\mu(y)\\
&\geq&\int_M \frac{C}{\mu(B_{2k})}\exp\Big(-\frac{d^2(x,y)}{3\cdot 2^{2k}t}\Big)w^{-1/(p-1)}(y)\,\d\mu(y)\\
&=&\int_M \frac{C}{\mu(B_0)}\Big(\frac{\mu(B_0)}{\mu(B_{2k})}\Big)\exp\Big(-\frac{d^2(x,y)}{3\cdot 2^{2k}t}\Big)w^{-1/(p-1)}(y)\,\d\mu(y)\\
&\geq&\int_M 2^{-2kN}\frac{C}{\mu(B_0)}\exp\Big(-\frac{d^2(x,y)}{3\cdot 2^{2k}t}\Big)w^{-1/(p-1)}(y)\,\d\mu(y)\\
&\geq&C2^{-2kN}P_t\big(w^{-1/(p-1)}\big)(x).
\end{eqnarray*}
Thus, for any $x\in M$ and any $t>0$,
\begin{eqnarray*}
&&P_tw(x)\big(P_t(w^{-1/(p-1)})(x)\big)^{p-1}\\
&\leq&C\big(P_t(w^{-1/(p-1)})(x)\big)^{p-1}\sum_{k=0}^\infty 2^{kN}\exp\Big(-\frac{2^{2k-2}}{5}\Big)\Big(\frac{2^{2kN}}{P_t\big(w^{-1/(p-1)}\big)(x)}\Big)^{p-1}\\
&=&C\sum_{k=0}^\infty 2^{(2p-1)kN}\exp\Big(-\frac{2^{2k-2}}{5}\Big)\leq C,
\end{eqnarray*}
which proves the first inequality of \eqref{comparison-of-weights-0}.

Now we assume that $w\in A_p^{heat}(M)$ and prove the second inequality in \eqref{comparison-of-weights-0}. Indeed, for any $t>0$ and $x\in M$, we can choose a constant $C>0$, depending on $N$, such that
$$\mu(B(x,\sqrt{t}))^{-1}\mathbf{1}_{B(x,\sqrt{t})}(\cdot)\leq C\mu(B(x,\sqrt{t}))^{-1}\exp[-d^2(x,\cdot)/(5t)]\leq C p_t(x,\cdot),$$
and hence,
\begin{eqnarray*}
&&\Big(\frac{1}{\mu(B(x,\sqrt{t}))}\int_{B(x,\sqrt{t})}w\,\d\mu\Big)\Big(\frac{1}{\mu(B(x,\sqrt{t}))}\int_{B(x,\sqrt{t})}w^{-1/(p-1)}\,\d\mu\Big)^{p-1}\\
&\leq&C\Big(\int_Mp_t(x,y)w(y)\,\d\mu(y)\Big)\Big(\int_Mp_t(x,y)w(y)^{-1/(p-1)}\,\d\mu(y)\Big)^{p-1}\\
&=&CP_tw(x)\big(P_t(w^{-1/(p-1)})(x)\big)^{p-1},
\end{eqnarray*}
which implies that the second inequality of \eqref{comparison-of-weights-0} holds for some positive constant $c_2$ depending on $N$.
\end{proof}

\begin{corollary}
Under the assumption of Corollary \ref{weighted-area}, if $w\in A_2(M)$, then for every $f\in C_c(M)$, there exists  a constant $C>0$ such that
\begin{equation*}
\|f\|_{L^2_w(M,\mu)}\leq C\|w\|_{A_2(M)}^{1/2}\|\mathcal{H}_\ast(f)\|_{L^2_w(M,\mu)},
\end{equation*}
\begin{equation*}
\|\mathcal{H}_\ast(f)\|_{L^2_w(M,\mu)}\leq C\inf_{1<s<2}\Big(\frac{s}{2-s}\|w\|_{A_s(M)}\Big)^{1/2}\|f\|_{L^2_w(M,\mu)},
\end{equation*}
\begin{equation*}
\|\mathcal{H}_\ast(f)\|_{L^2_w(M,\mu)}\leq C\|w\|_{A_2(M)}\|f\|_{L^2_w(M,\mu)},
\end{equation*}
\begin{equation*}
\|\mathcal{H}(f)\|_{L^2_w(M,\mu)}\leq C\|w\|_{A_2(M)}\|f\|_{L^2_w(M,\mu)},
\end{equation*}
$$\|\mathcal{A}_\alpha(f)\|_{L^2_w(M,\mu)}\leq Ce^{\alpha^2/6}\|w\|_{A_2(M)}\|f\|_{L^2_w(M,\mu)}.$$
\end{corollary}
\begin{proof}
By Theorem \ref{main}, Corollary \ref{weighted-area} and Theorem \ref{comparison-of-weights}, we immediately prove the corollary.
\end{proof}

\section{Remarks}\hskip\parindent
In this section, we give some remarks to illustrate the main results  and to point out another interesting setting where the results should be established in a similar manner.

Let $M$ be a complete smooth Riemannian manifold, $d$ be the Riemannian distance, $\mu$ be the Riemannian volume measure and $\Ric$ be the Ricci curvature. Let $\Delta$ be the Laplace--Beltrami operator on $M$. The corresponding heat flow and heat kernel are still denoted by $(P_t)_{t\geq0}$ and $(p_t)_{t\geq0}$, respectively.  Then, applying the same approach as in Sections 3 and 4, we can prove the following theorem.
\begin{theorem}\label{Riem-manifold-eg}
Let $M$ be a complete and noncompact Riemannian manifold with dimension $n\geq2$. Suppose that $\Ric\geq0$ and $w\in A_2^{heat}(M)$. Then, all the inequalities in Theorem \ref{main} and Corollary \ref{weighted-area} hold true, as well as the result in Theorem \ref{comparison-of-weights}.
\end{theorem}

Note that the Riemannian manifold $(M,d,\mu)$ with $\Ric\geq0$ is a particular $\RCD^*(0,n)$ space. Hence, Theorem \ref{Riem-manifold-eg} is an immediate consequence of Proposition \ref{noncom-rcd-main}. We should mention that, the former part of  assertions in Theorem \ref{Riem-manifold-eg} is not new, which have been obtained recently in \cite[Theorem 6.1]{BO2016} without assuming the topological property that $M$ is noncompact but instead under an additional assumption on the heat kernel that
\begin{eqnarray}\label{large-kernel-asy}
\sup_{x\in M} p_t(x,x)=c_t\rightarrow0,\quad\mbox{as }t\rightarrow\infty.
\end{eqnarray}
However, from the heat kernel upper bound in \eqref{gassian-bound-1} and Lemma \ref{low-vol-bound}, we immediately see that \eqref{large-kernel-asy} holds under the assumption of Theorem \ref{Riem-manifold-eg}.

\medskip

Finally, we remark that results in Theorem \ref{main}, Proposition \ref{noncom-rcd-main}, Corollary \ref{weighted-area} and Theorem \ref{comparison-of-weights} should be similarly established on sub-Riemannian manifolds satisfying the generalized curvature-dimension condition  $\CD(0,\rho_2, \kappa, m)$ with $\rho_2>0$, $\kappa\geq0$ and $2\leq m<\infty$, in the sense of Baudoin--Garofalo \cite{BaudoinBonnefont2,BaudoinGarofalo2017}, although, generally speaking, in that setting, the curvature-dimension dimension condition in the sense of Lott--Sturm--Villani is not available (see e.g. \cite{Juillet}).

\subsection*{Acknowledgment}\hskip\parindent
The author wishes to thank the referee for his or her very careful reading and quite helpful corrections and suggestions of the manuscript.  The author would also like to acknowledge the financial support from the National Natural Science Foundation of China (Nos. 11401403, 11571347 and 11831014).


\begin{thebibliography}{a23}
\bibitem{agmr2015} L. Ambrosio, N. Gigli, A. Mondino, T. Rajala,
Riemannian Ricci curvature lower bounds in metric measure spaces with $\sigma$-finite measure, Trans. Amer. Math. Soc. 367 (2015), 4661--4701.

\bibitem{AmbrosioGigliSavare2005} L. Ambrosio, N. Gigli, G. Savar\'e, \emph{Gradient flows in metric spaces and in the
space of probability measures}, Lectures in Mathematics ETH Z\"{u}rich, Birkh\"{a}user Verlag, Basel, 2008.

\bibitem{AmbrosioGigliSavare2014} L. Ambrosio, N. Gigli, G. Savar\'{e}, Calculus and heat flow in metric measure spaces and applications to spaces with Ricci bounds from below, Invent. Math. 195 (2014), 289--391.

\bibitem{AmbrosioGigliSavare2011b} L. Ambrosio, N. Gigli, G. Savar\'e, Metric measure spaces with Riemannian Ricci curvature bounded from below, Duke Math. J. 163 (7) (2014), 1405--1490.

\bibitem{AmbrosioGigliSavare2012} L. Ambrosio, N. Gigli, G. Savar\'e, Bakry--\'{E}mery curvature-dimension condition and Riemannian Ricci curvature bounds, Ann. Prob. 43 (1) (2015), 339--404.

\bibitem{AMS2016} L. Ambrosio, A. Mondino, G. Savar\'{e}, On the Bakry--\'{E}mery condition, the gradient estimates and the Local-to-Global property of $\RCD^\ast(K,N)$ metric measure spaces, J. Geom. Anal. 26 (2016), 24--56.

\bibitem{ASZ2009}
L. Ambrosio, G. Savar\'{e}, L. Zambotti,  Existence and stability for Fokker--Planck equations with log-concave reference measure, Probab. Theory Related Fields, 145 (2009), 517--564.

\bibitem{BacherandSturm2010} K. Bacher, K.-T. Sturm, Localization and tensorization properties of the curvature-dimension condition for metric
measure spaces. J. Funct. Anal. 259 (1) (2010), 28--56.

\bibitem{BakryEmery1985} D. Bakry, M. Emery, Diffusions hypercontractives,  S\'{e}minaire de Probabilit\'{e}s XIX, 1983/84,
Lecture Notes in Math. 1123 (1985),  177--206.

\bibitem{BM2003}
R. Ba\~{n}uelos, P.J. M\'{e}ndez-Hern\'{a}ndez, Space-time Brownian motion and the Beurling--Ahlfors
transform, Indiana Univ. Math. J. 52 (2003), 981--990.

\bibitem{BO2016} R. Ba\~{n}uelos, A. Osekowski, WEIGHTED $L^2$ INEQUALITIES FOR SQUARE FUNCTIONS, Trans. Amer. Math. Soc., 370  (2018), 2391--2422.

\bibitem{BaudoinBonnefont2} F. Baudoin, M. Bonnefont, N. Garofalo, A sub-Riemannian curvature-dimension inequality, volume doubling property and the Poincar\'{e} inequality, Math. Ann. 358 (2014), 833--860.

\bibitem{BaudoinGarofalo2017} F. Baudoin, N. Garofalo, Curvature-dimension inequalities and Ricci lower bounds for sub-Riemannian manifolds with transverse symmetries, J. Eur. Math. Soc. 19 (2017), 151--219.

\bibitem{Buckley1993}
S.M. Buckley, Estimates for operator norms on weighted spaces and reverse Jensen inequalities, Trans. Amer. Math.
Soc. 340 (1993), 253--272.

\bibitem{Calabi1975}
E. Calabi, On manifolds with non-negative Ricci curvature II, Notices Amer. Math. Soc. 22 (1975), A205.

\bibitem{CD2003}
T. Coulhon, X.T. Duong, Riesz transform and related inequalities on non-compact Riemannian
manifolds, Comm. Pure Appl. Math.  56 (2003), 1728--1751.

\bibitem{CDD}
T. Coulhon, X.T. Duong,  X.-D. Li, Littlewood--Paley--Stein functions on complete Riemannian manifolds for $1\le p\le 2$, Studia Math. 154 (2003), 37--57.

\bibitem{CMP2010}
D. Cruz-Urube, J. Martell, C. Perez, Sharp weighted estimates for approximating dyadic operators,  arXiv:1001.4724v2.

\bibitem{CMP2012}
D. Cruz-Urube, J. Martell, C. Perez, Sharp weighted estimates for classical operators, Adv. Math. 229  no. 1 (2012), 408--441.

\bibitem{DomPet2016}
K. Domelevo, S. Petermichl, Continuous-time  sparse  domination, arXiv:1607.06319v2.

\bibitem{DomPet2017} K. Domelevo, S. Petermichl, Differential subordination under change of law, Ann. Prob. 47 (2) (2019), 896--925.

\bibitem{eks2013} M. Erbar, K. Kuwada, K.-T. Sturm, On the equivalence of the entropic curvature-dimension condition and Bochner's inequality
on metric measure spaces, Invent. Math. 201 (2015), 993--1071.

\bibitem{FOT2011}
M. Fukushima, Y. Oshima, M. Takeda, {\emph{Dirichlet Forms and Symmetric Markov Processes},} de Gruyter Studies in Math. vol. 19, Walter de Gruyter, Berlin, 2nd rev. and ext. ed., 2011.

\bibitem{gi2012} N. Gigli, {\emph{On the differential structure of metric measure spaces and applications},} Mem. Amer. Math.
Soc. 236, 2015.

\bibitem{Hyt2012}
T. Hyt\"{o}nen, The sharp weighted bound for general Calder\'{o}n--Zygmund operators, Ann. of Math. 175 (2012), 1473--1506.

\bibitem{HPTV2014}
T. Hyt\"{o}nen, C. P\'{e}rez, S. Treil, A. Volberg, Sharp weighted estimates for dyadic shifts and the $A_2$ conjecture,  J. Reine Angew. Math. 687 (2014), 43--86.

\bibitem{IzumisawaKazamaki1977} M. Izumisawa, N. Kazamaki, Weighted norm inequalities for martingales, T\^{o}hoku Math. Journ. 29
(1977), 115--124.

\bibitem{JLZ2014} R. Jiang, H. Li, H. Zhang, Heat Kernel Bounds on Metric Measure Spaces and Some Applications, Potential Anal. 44 (2016), 601--627.

\bibitem{Juillet} N. Juillet, Geometric inequalities and generalized Ricci bounds in the Heisenberg group, Int. Math. Res. Not. 13  (2009), 2347--2373.

\bibitem{LPR2010} M. Lacey, S. Petermichl, M. Reguera, Sharp $A_2$ inequality for Haar shift operators, Math. Ann. 348 (1) (2010), 127--141.

\bibitem{Lerner2013} A. K. Lerner, A simple proof of the $A_2$ conjecture, International Mathematics Research Notices (2013) No. 14, 3159--3170.

\bibitem{LiH2017} H. Li, Sharp heat kernel bounds and entropy in metric measure spaces, Sci. China Math. Vol. 61  No. 3 (2018), 487--510.

\bibitem{LiH2017+} H. Li, Littlewood--Paley--Stein inequalities on $\RCD(K,\infty)$ spaces, arXiv:1905.01432.

\bibitem{LiWang2016}
H. Li, J. Wang, Littlewood--Paley--Stein estimates for non-local Dirichlet forms, arXiv:1704.02690v3. To appear in J. d'Analyse Math..

\bibitem{LP1931} J.E. Littlewood, R.E. Paley, Theorems on Fourier series and power series (I), J. London Math. Soc. 6 (1931), 230--233.

\bibitem{Lou1987} N. Lohou\'{e},  Estimations des fonctions de Littlewood--Paley--Stein sur les vari\'{e}t\'{e}s
riemanniennes \`{a} courbure non positive,  Ann. Sci. Ecole Norm. Sup. 20 (1987), 505--544.

\bibitem{LV2009} J. Lott, C. Villani, Ricci curvature for metric-measure spaces via optimal transport,  Ann. of Math. (2)  169  (2009), 903--991.

\bibitem{Mey}
P.-A. Meyer,  D\'{e}monstration probabiliste de certaines in\'{e}galiti\'{e}s de Littlewood-Paley. Expos\'{e} II: l'op\'{e}rateur carr\'{e} du champ, S\'{e}minaire de probabilit\'{e}s   10  (1976), 142--161.

\bibitem{Mey1981}
P.-A. Meyer, Retour sur la th\'{e}orie de Littlewood--Paley,  S\'{e}minaire de Probabilit\'{e}s XV,
Lecture Notes in Math.  850 (1981), 151--166.

\bibitem{Muckenhoupt72}
B. Muckenhoupt, Weighted norm inequalities for the Hardy maximal function, Trans.
Amer. Math. Soc. 165 (1972), 207--226.

\bibitem{Ohta2007}
Shin-ichi Ohta, On the measure contraction property of metric measure spaces, Comment. Math.
Helv. 82 (2007), 805--828.

\bibitem{Petermichl2007}
S. Petermichl, The sharp bound for the Hilbert transform on weighted Lebesgue spaces in terms of the classical $A_p$
characteristic, Amer. J. Math. 129 (5) (2007), 1355--1375.

\bibitem{Petermichl2008}
S. Petermichl, The sharp weighted bound for the Riesz transforms, Proc. Amer. Math. Soc. 136 (4) (2008), 1237--1249.

\bibitem{PV2002}
S. Petermichl, A. Volberg, Heating of the Ahlfors--Beurling operator: weakly quasiregular maps on the
plane are quasiregular, Duke Math. J. 112 (2002), 281--305.

\bibitem{Petr2011}
A. Petrunin, Alexandrov meets Lott--Villani--Sturm, M\"{u}nster J. Math. 4 (2011), 53--64.

\bibitem{St1958}
E.M. Stein, On the functions of Littlewood--Paley, Lusin, and Marcinkiewicz, Trans. Amer. Math. Soc. 88 (1958), 430--466.

\bibitem{St1970}
E.M. Stein, \emph{Singular Integrals and Differentiability Properties of Functions}, Princeton Mathematical Series, No. 30, Princeton University Press, Princeton, 1970.

\bibitem{Stein70}
E.M. Stein, \emph{Topics in Harmonic Analysis Related to the Littlewood--Paley Theory}, Ann. of Math. Stud.  63, Princeton Univ. Press, Princeton, 1970.

\bibitem{Sturm2006a} K.-T. Sturm, On the geometry of metric measure spaces. I, Acta Math. 196 (2006), 65--131.

\bibitem{Sturm2006b} K.-T. Sturm, On the geometry of metric measure spaces. II, Acta Math. 196 (2006), 133--177.

\bibitem{Yau1975} S.T. Yau, Some function theoretic properties of complete Riemannian manifolds and their applications to geometry, Indiana Univ. Math. J. 25 (1976), 659--670.

\bibitem{ZZ2011}
H.-C. Zhang, X.-P. Zhu, Ricci curvature on Alexandrov spaces and rigidity theorems, Comm. Anal. Geom. 18 (3), 503--553.
\end{thebibliography}
\end{document}